\documentclass[a4paper]{amsart}
\title[The duality on wave fronts]{%
   The duality between singular points and\\
   inflection points on wave fronts}
\date{May 11, 2010}

\usepackage[dvips]{graphicx} 
\usepackage{verbatim,enumerate}
\usepackage{amssymb}

\usepackage{amsthm}
\theoremstyle{plain}
 \newtheorem{theorem}{Theorem}[section]
 \newtheorem{introtheorem}{Theorem}
 \renewcommand{\theintrotheorem}{\Alph{introtheorem}}
 \newtheorem{introproposition}[introtheorem]{Proposition}
 \newtheorem{introcorollary}[introtheorem]{Corollary}
 \newtheorem*{theorem*}{Theorem}
 \newtheorem*{lemma*}{Lemma}
 \newtheorem{proposition}[theorem]{Proposition}
 \newtheorem{fact}[theorem]{Fact}
 \newtheorem{fact*}{Fact}

\theoremstyle{remark}
 \newtheorem{definition}[theorem]{Definition}
 \newtheorem{remark}[theorem]{Remark}
 \newtheorem*{remark*}{Remark}
 
 \newtheorem{example}[theorem]{Example}
 \newtheorem*{added}{Added in Proof}
\numberwithin{equation}{section}
\numberwithin{figure}{section}
\renewcommand{\theenumi}{{\rm(\arabic{enumi})}}
\renewcommand{\labelenumi}{\theenumi}
\newcommand{\R}{\boldsymbol{R}}
\newcommand{\C}{\boldsymbol{C}}
\newcommand{\K}{\boldsymbol{K}}
\newcommand{\zv}{\boldsymbol{0}}
\newcommand{\G}{\mathcal{G}}

\newcommand{\trans}[1]{{\vphantom{#1}}^t{#1}}
\renewcommand{\phi}{\varphi}
\newcommand{\pmt}[1]{{\begin{pmatrix} #1  \end{pmatrix}}}
\author{Kentaro Saji}
\address[Saji]{
  Department of Mathematics,
  Faculty of Educaton,
  Gifu University,  Yanagido 1-1, Gifu 501-1193, Japan}
\email{ksaji@gifu-u.ac.jp}

\author{Masaaki Umehara}
\address[Umehara]{%
   Department of Mathematics, Graduate School of Science,
   Osaka University,
   Toyonaka, Osaka 560-0043,
   Japan
}
\email{umehara@math.sci.osaka-u.ac.jp}

\author{Kotaro Yamada}
\address[Yamada]{%
Department of Mathematics,
   Tokyo Institute of Technology,
   O-okayama, Meguro, Tokyo 152-8551, Japan}
\email{kotaro@math.titech.ac.jp}
\begin{document}
\maketitle
\begin{abstract}
 In the previous paper,
 the authors gave criteria for $A_{k+1}$-type singularities on
 wave fronts. 
 Using them, we show in this paper that there is a duality between
 singular points and inflection points on wave fronts  in the projective
 space.
 As an application, we show that the algebraic sum of $2$-inflection
 points  (i.e.\ godron points)
 on an immersed surface in the real projective space is equal to the
 Euler number of $M_-$. 
 Here $M^2$ is a compact orientable 2-manifold, and $M_-$ is the open
 subset of $M^2$ where
 the Hessian of $f$ takes  negative values. 
 This is a generalization of Bleecker and Wilson's formula \cite{BW}
 for immersed surfaces in the affine $3$-space.
\end{abstract}

%%%%%%%%%%%%%%%%%%%%%%%%%%%%%%%%%%%%%%%%%%%%%%%%
\section{Introduction}
We denote by $\K$ the real number field $\R$ or the complex number field
$\C$. Let $n$ and $m$ be positive integers.
A map $F\colon{}\K^n\to \K^m$ is called {\em $\K$-differentiable\/}
if it is a $C^\infty$-map when $\K=\R$, 
and is a holomorphic map when $\K=\C$.
Throughout this paper, 
we denote by $P(V)$  the $\K$-projective space associated 
to a vector space $V$ over $\K$ and let
$\pi:V\to P(V)$ be the canonical projection.

Let $M^{n}$ and $N^{n+1}$ be 
$\K$-differentiable manifolds of dimension $n$ 
and of dimension $n+1$, 
respectively.
The projectified $\K$-cotangent bundle
\[
  P(T^*N^{n+1}):=\bigcup_{p\in N^{n+1}} P(T^*_pN^{n+1})
\]
 has a canonical $\K$-contact structure.
A $\K$-differentiable map $f : M^n
\to N^{n+1}$ is called a {\it frontal} 
if $f$ lifts to 
a $\K$-isotropic map $L_f$, i.e.,
a $\K$-differentiable map $L_f : M^n \to P(T^*N^{n+1})$
such that the image $dL_f(TM^n)$ of the $\K$-tangent bundle 
$TM^n$ lies in the contact hyperplane 
field on $P(T^*N^{n+1})$.
Moreover, $f$ is called a {\it wave front} or a {\it front} if 
it lifts to a $\K$-isotropic immersion $L_f$.
(In this case, $L_f$ is called a {\it Legendrian immersion}.)
Frontals (and therefore fronts) generalize 
immersions, as they allow for
singular points. 
A frontal $f$ is said to be {\it co-orientable}
if its $\K$-isotropic lift $L_f$ can lift up to a
$\K$-differentiable map into the $\K$-cotangent 
bundle $T^*N^{n+1}$, 
otherwise it is said to be {\it non-co-orientable}.
It should be remarked that, when $N^{n+1}$ is a Riemannian
manifold, 
a front $f$ is co-orientable if and only if 
there is a globally defined unit normal vector field 
$\nu$ along $f$. 

Now we set $N^{n+1}=\K^{n+1}$.
Suppose that a 
$\K$-differentiable map
$F:M^{n}\to \K^{n+1}$ is a frontal. 
Then, for each $p\in M^{n}$,
there exist a neighborhood $U$ of $p$ and a map
\[
   \nu:U\longrightarrow (\K^{n+1})^*\setminus \{\zv\}
\]
into the dual vector space $(\K^{n+1})^*$ of $\K^{n+1}$
such that the canonical pairing $\nu\cdot dF(v)$ vanishes 
for any $v\in TU$. 
We call $\nu$ a {\it local normal map\/} of the frontal $F$.
We set $\G:=\pi\circ\nu$, which is called a (local) {\it Gauss map\/} of
$F$. 
In this setting, $F$ is a front if and only if
\[
    L:=(F,\G):U\longrightarrow \K^{n+1}\times P\bigl((\K^{n+1})^*
\bigr)
\]
is an immersion. 
When $F$ itself is an immersion, it is, of course, a front.
If this is the case, for a fixed  local $\K$-differentiable
coordinate system $(x^1,\dots,x^{n})$ on $U$, 
we set
\begin{equation}\label{eq:nu}
 \nu_p:\K^{n+1}\ni v \longmapsto 
   \det(F_{x^1}(p),\dots,F_{x^{n}}(p),v)
   \in \K
   \qquad (p\in U),
\end{equation}
where $F_{x^j}:=\partial F/\partial x^j$ ($j=1,\dots,n$)
and \lq$\det$\rq\ is the determinant function on $\K^{n+1}$.
Then we get a $\K$-differentiable map
$
    \nu:U\ni p \longmapsto \nu_p\in (\K^{n+1})^*,
$
which gives a local normal map of $F$.

Now, we return to the case that $F$ is a front.
Then it is well-known that
the local Gauss map $\mathcal G$ 
induces a global map
\begin{equation}\label{eq:G}
     \G:M^{n}\longrightarrow P\bigl((\K^{n+1})^*\bigr)
\end{equation}
which is called the {\it affine Gauss map\/} of $F$.
(In fact, the Gauss map $\G$ depends only on the affine structure of 
$\K^{n+1}$.)

We set 
\begin{equation}\label{eq:h}
   h_{ij}:=\nu\cdot F_{x^ix^j}=-\nu_{x^i} \cdot F_{x^j}
      \qquad (i,j=1,\dots,n),
\end{equation}
where $\cdot$ is the canonical pairing between $\K^{n+1}$ and
$(\K^{n+1})^*$, and 
\[
  F_{x^ix^j}=\frac{\partial^2 F}{\partial x^i\partial x^j},
\quad 
  F_{x^j}=\frac{\partial F}{\partial x^j},\quad
  \nu_{x^i}=\frac{\partial \nu}{\partial x^i}.
\]
Then
\begin{equation}\label{eq:H}
    H:=\sum_{i,j=1}^{n}h_{ij}dx^i\, dx^j
     \qquad \left(
	     dx^i\, dx^j:=\frac12(dx^i\otimes dx^j+dx^j\otimes dx^i)
	    \right)
\end{equation}
gives a $\K$-valued symmetric tensor on $U$, which is called
the {\it Hessian form\/} 
of $F$ associated to $\nu$.
Here, the $\K$-differentiable function
\begin{equation}\label{eq:hessian1}
   h:=\det(h_{ij}):U\longrightarrow \K
\end{equation}
is called the {\it Hessian\/} of $F$.
A point $p\in M^{n}$ is called an {\it inflection point\/} of 
$F$ if it belongs to the zeros of $h$.
An inflection point $p$ is called {\it nondegenerate\/} if
the derivative $dh$ does not vanish at $p$.
In this case,  the set of inflection points  $I(F)$ 
consists of an embedded $\K$-differentiable hypersurface
of $U$ near $p$ and there exists a non-vanishing $\K$-differentiable
vector field $\xi$ along $I(F)$ such that $H(\xi,v)=0$ for all 
$v\in TU$.
Such a vector field $\xi$ is called an {\it asymptotic vector field\/}
along $I(F)$, 
and $[\xi]=\pi(\xi)\in P(\K^{n+1})$ is called the 
{\it asymptotic direction}.
It can be easily checked that the definition of inflection 
points and the nondegeneracy of inflection points are independent of
choice of $\nu$ and a local coordinate system. 

In Section~\ref{sec:prelim},
we shall define the terminology that
\begin{itemize}
\item a $\K$-differentiable vector field
$\eta$ along a $\K$-differentiable hypersurface $S$ of $M^n$
is {\it $k$-nondegenerate at} $p\in S$, and
\item
{\it $\eta$ meets $S$ at $p$ with multiplicity $k+1$.}
\end{itemize}
Using this new terminology, $p(\in I(F))$ is called an {\it
$A_{k+1}$-inflection point} if 
$\xi$ is $k$-nondegenerate at $p$ but does not meet $I(F)$ 
with multiplicity $k+1$.
In Section~\ref{sec:prelim}, we shall prove the following:
\begin{introtheorem}\label{thm:A}
 Let $F:M^{n}\to \K^{n+1}$ be an immersed 
 $\K$-differentiable hypersurface.
 Then $p\in M^{n}$ is an $A_{k+1}$-inflection point $(1\leq k\leq n)$
 if and only if the affine Gauss map $\mathcal G$
 has an $A_k$-Morin singularity at $p$.
 $($See the appendix of \cite{SUY3}
    for the definition of $A_k$-Morin singularities,
    which corresponds to $A_{k+1}$-points under the intrinsic
    formulation of singularities as in the reference given in
    Added in Proof.
 $)$
\end{introtheorem}

Though our definition of $A_{k+1}$-inflection points are given
in terms of the Hessian, this assertion allows us to define
$A_{k+1}$-inflection points by the singularities
of their affine Gauss map, which might be more familiar to 
readers than our definition.
However, the new notion ``$k$-multiplicity'' introduced in the present
paper is very useful for recognizing the duality between singular points
and inflection points.
Moreover, as mentioned above, 
our definition of $A_k$-inflection points works even when $F$ is a
front. 
We have the following dual assertion for the previous theorem.
Let $\G\colon{}M^n \to P\bigl((\K^{n+1})^*\bigr)$ be an immersion.
Then $p\in M^n$ is an inflection point of 
$\nu\colon{}M^n \to(\K^{n+1})^*$ such that $\pi\circ\nu=\G$.
This property does not depend on a choice of $\nu$.

\begingroup
\addtocounter{introtheorem}{-1}
\renewcommand{\theintrotheorem}{$\mathbf{\Alph{introtheorem}}'$}
\begin{introproposition}
\label{prop:A}
 Let $F:M^n\to \K^{n+1}$ be a front.
 Suppose that the affine Gauss map  $\G:M^n\to P((\K^{n+1})^*)$ 
 is a $\K$-immersion.
 Then $p\in M^{n}$ is an $A_{k+1}$-inflection point 
 of $\G$ $(1\leq k\leq n)$ if and only if $F$
 has an $A_{k+1}$-singularity at $p$. 
 {\rm (}See $(1.1)$ in \cite{SUY3} for  the definition of
 $A_{k+1}$-singularities.{\rm )}
\end{introproposition}
\endgroup

In the case that $\K=\R$, $n=3$ and $F$ is an immersion, 
an $A_{3}$-inflection point is 
known as a {\em cusp of the Gauss map}
(cf. \cite{BGM}).

It can be easily seen that inflection points and the asymptotic
directions are invariant under projective transformations.
So we can define $A_{k+1}$-inflection points ($1\leq k\leq n$) of
an immersion $f:M^{n}\to P(\K^{n+2})$.
For each $p\in M^{n}$, we take a local
$\K$-differentiable coordinate system
$(U;x^1,\dots,x^{n})(\subset M^n)$.
Then there exists a $\K$-immersion $F:U\to \K^{n+2}$ such that $f=[F]$ is
the projection of $F$.
We set
\begin{equation}\label{eq:FG1}
 G:U\ni p \longmapsto 
  F_{x^1}(p)\wedge 
  F_{x^2}(p)\wedge 
  \cdots \wedge F_{x^n}(p)
  \wedge F(p)\in (\K^{n+2})^*.
\end{equation}
Here, we identify $(\K^{n+2})^*$ with $\bigwedge^{n+1}\K^{n+2}$ by
\[
\bigwedge\nolimits^{n+1} \K^{n+2}
     \ni v_1\wedge \cdots \wedge v_{n+1}
     \quad
     \longleftrightarrow 
     \quad
     \det(v_1, \dots, v_{n+1},*)\in (\K^{n+2})^*,
\]
where
\lq$\det$\rq\ is the determinant function on $\K^{n+2}$.
Then $G$ satisfies
\begin{equation}\label{eq:FG2}
     G \cdot F=0,\qquad G\cdot dF=dG\cdot F=0,
\end{equation}
where $\cdot$ is the canonical pairing between $\K^{n+2}$ and
$(\K^{n+2})^*$.
Since, $g:=\pi\circ G$ does not depend on the choice of 
a local coordinate system,
the projection of $G$ induces a globally defined $\K$-differentiable map
\[
    g=[G]:M^n\to P\bigl((\K^{n+2})^*\bigr),
\]
which is called the {\it dual\/} front of $f$.
We set
\[
   h:=\det(h_{ij}):U\longrightarrow \K
            \qquad \bigl(h_{ij}:=G\cdot F_{x^ix^j}
	            =-G_{x^i}\cdot F_{x^j}\bigr),
\]
which is called the {\it Hessian\/} of $F$.
The inflection points of $f$ correspond to the zeros of $h$.  

In Section~\ref{sec:dual}, we prove the following
\begin{introtheorem}
\label{thm:B}
 Let $f:M^{n}\to P(\K^{n+2})$ be an 
 immersed $\K$-differentiable hypersurface.
 Then $p\in M^{n}$ is an $A_{k+1}$-inflection point $(k\leq n)$
 if and only if the dual front $g$ has an $A_k$-singularity at $p$.
\end{introtheorem}

Next, we consider the case of $\K=\R$.
In \cite{SUY1}, we defined the {\it tail part\/} of 
a swallowtail, that is, an $A_3$-singular point. 
An $A_3$-inflection point $p$ of $f:M^{2}\to P(\R^4)$
is called {\it positive\/} (resp.\ {\it negative}), 
if the Hessian takes negative (resp.\ positive) values
on the tail part of the dual of $f$ at $p$.
Let $p\in M^2$ be an  $A_3$-inflection point.
Then there exists a neighborhood $U$ such that $f(U)$ 
is contained in an
affine space $A^3$ in $P(\R^4)$.
Then the affine Gauss map $\G:U\to P(A^3)$ has an elliptic cusp 
(resp.\ a hyperbolic cusp)
if and only if it is positive (resp.\ negative)
(see \cite[p. 33]{BGM}).
In \cite{uribe}, Uribe-Vargas introduced a projective invariant $\rho$
and studied the projective geometry of swallowtails.
He proved that an $A_3$-inflection point is positive (resp.\ negative)
if and only if $\rho>1$ (resp.\ $\rho<1$).
The property that $h$ as in \eqref{eq:hessian1}
is negative is also independent of the choice of a
local coordinate system.
So we can define the set of negative points
\[
     M_{-}:=\{p\in M^{2}\,;\, h(p)<0\}.
\]
In Section~\ref{sec:dual},
we shall prove the following assertion as an application.
\begin{introtheorem}
\label{thm:C}
  Let $M^2$ be a compact orientable $C^\infty$-manifold without
 boundary, and
 $f:M^{2}\to P(\R^4)$ an immersion.
 We denote by $i^+_2(f)$ {\rm(}resp.\ $i^-_2(f)${\rm)} the number of
 positive $A_3$-inflection points 
 {\rm (}resp.\ negative $A_3$-inflection points{\rm)} on $M^2$
 {\rm(}see Section~\ref{sec:dual} for the precise definition 
     of $i^+_2(f)$ and $i^-_2(f)${\rm)}.
 Suppose that inflection points of $f$ 
 consist only of $A_2$ and $A_3$-inflection points.
 Then the following identity holds
 \begin{equation}\label{eq:BW}
      i_2^+(f)-i_2^-(f)=2\chi(M_-).
 \end{equation}
\end{introtheorem}

The above formula is a generalization of that of 
Bleecker and Wilson \cite{BW} when $f(M^2)$ is contained
in an affine $3$-space

\begin{introcorollary}[Uribe-Vargas {\cite[Corollary 4]{uribe}}]
\label{cor:1}
 Under the assumption of Theorem~\ref{thm:C},
 the total number $i_2^+(f)+i_2^-(f)$ of $A_3$-inflection points is
 even.
\end{introcorollary}

In \cite{uribe}, this corollary is proved by counting the parity of 
a loop consisting of flecnodal curves which bound two $A_3$-inflection
points.
\begin{introcorollary}
\label{cor:2}
 The same formula \eqref{eq:BW} holds for
 an immersed surface in 
 the unit $3$-sphere $S^3$ or in the hyperbolic $3$-space $H^3$.
\end{introcorollary}
\begin{proof}
 Let $\pi:S^3\to P(\R^4)$ be the canonical projection.
 If  $f:M^2\to S^3$ is an immersion,
 we get the assertion
applying  Theorem~\ref{thm:C} to $\pi\circ f$.
 On the other hand, if 
 $f$ is an immersion into $H^3$,
 we consider the canonical projective embedding
 $\iota:H^3\to S^3_+$
 where  $S^3_+$ is the open hemisphere of $S^3$.
 Then we get the assertion applying  
Theorem~\ref{thm:C} to $\pi\circ \iota\circ f$.
\end{proof}

Finally,
in Section~\ref{sec:cusp}, we shall introduce a new invariant for
$3/2$-cusps using the duality, 
which is a measure for acuteness using the classical cycloid.  

\bigskip
This work is inspired by the result of
Izumiya, Pei and Sano \cite{ips} that characterizes
$A_2$ and $A_3$-singular points on surfaces in $H^3$ via the singularity
of certain height functions,
and the result on the duality between space-like surfaces
in hyperbolic $3$-space (resp.\ in light-cone),
and those in de Sitter space (resp.\ in light-cone) given
by Izumiya \cite{iz}.
The authors would like to thank Shyuichi Izumiya
for his impressive informal talk at Karatsu, 2005.

\section{Preliminaries and a proof of Theorem~\ref{thm:A}}
\label{sec:prelim}

In this section, we shall introduce a new notion
``multiplicity'' for a contact of a given vector field along an immersed
hypersurface. 
Then our previous criterion 
for $A_k$-singularities (given in \cite{SUY3}) can be generalized to
the criteria for $k$-multiple contactness of a given vector field
(see Theorem~\ref{thm:actual}).

Let $M^n$ be a $\K$-differentiable manifold and $S(\subset M^n)$ 
an embedded $\K$-differentiable hypersurface in $M^n$.
We fix $p\in S$ and take a $\K$-differentiable vector field 
\[
    \eta:S\supset V\ni q \longmapsto \eta_q\in T_qM^n
\]
along  $S$ defined on a neighborhood $V\subset S$ of $p$.
Then we can construct a $\K$-differential vector field $\tilde \eta$
defined on a neighborhood $U\subset M^n$ of 
$p$
such that the restriction $\tilde \eta|_S$ coincides with $\eta$.
Such an $\tilde \eta$ is called {\it an extension of $\eta$}.
(The local existence of $\tilde \eta$ is mentioned in  
 \cite[Remark 2.2]{SUY3}.)

\begin{definition}\label{def:admissible}
 Let $p$ be an arbitrary point on $S$,
and $U$ a neighborhood of $p$ in $M^n$.
A  $\K$-differentiable function $\phi:U\to \K$ 
is called {\it admissible} near $p$
if it satisfies the following properties
 \begin{enumerate}
  \item $O:=U\cap S$ is the zero level set of $\phi$, and
  \item $d\phi$ never vanishes on $O$.
 \end{enumerate}
\end{definition} 
One can easily find an admissible function near $p$.
We set
$\phi':=d\phi(\tilde\eta):U\to \K$
and define a subset $S_2(\subset O\subset S)$ by
\[
   S_2:=\{q\in O \,;\,\phi'(q)=0\}=\{q\in O\,;\, \eta_q\in T_qS\}.
\]
If $p\in S_2$, then $\eta$ is said to 
{\it meet $S$ with multiplicity $2$ at $p$\/}
or equivalently, $\eta$ is said to 
{\it contact $S$ with multiplicity $2$ at $p$}.
Otherwise, $\eta$ is said to {\it meet $S$ with multiplicity $1$ at $p$}.
Moreover, if $d\phi'(T_pO)\ne \{\zv\}$,
$\eta$ is said to {\it be $2$-nondegenerate at $p$}.
The $k$-th multiple contactness and $k$-nondegeneracy are defined
inductively.
In fact, if the $j$-th multiple contactness and 
the submanifolds  $S_j$
have been already defined for $j=1,\dots,k$ ($S_1=S$),
we set
\[
    \phi^{(k)}:=d\phi^{(k-1)}(\tilde\eta):
           U\longrightarrow \K
           \qquad (\phi^{(1)}:=\phi')
\]
and can define a subset of $S_{k}$ by
\[
    S_{k+1}:=\{q\in S_k\,;\,\phi^{(k)}(q)=0\}
          =\{q\in S_k\,;\, \eta_q\in T_qS_k\}.
\]
We say that
 $\eta$  {\it meets $S$ with multiplicity $k+1$ at $p$\/}
if $\eta$ is $k$-nondegenerate at $p$ and
$p\in S_{k+1}$. 
Moreover, if $d\phi^{(k)}(T_pS_{k})\ne \{\zv\}$,
$\eta$ is called {\it $(k+1)$-nondegenerate\/} at $p$.
If $\eta$ is $(k+1)$-nondegenerate\ at $p$, then
$S_{k+1}$ is a hypersurface of $S_k$ near $p$.

\begin{remark}\label{rmk:add}
Here we did not define \lq $1$-nondegeneracy\rq\
of $\eta$. However, from now on, 
{\it any $\K$-differentiable vector field
$\eta$ of $M^n$ along $S$ is always
$1$-nondegenerate by convention}.
In the previous paper \cite{SUY3}, 
\lq $1$-nondegeneracy\rq\ (i.e. nondegeneracy) is
defined not for a vector field along the singular set
but for a given singular point. 
If a singular point $p\in U$ of a front $f:U\to \K^{n+1}$
is nondegenerate in the sense of
\cite{SUY3}, then the function $\lambda:U\to \K$ defined in
\cite[(2.1)]{SUY3} is an admissible function, and the
null vector field $\eta$ along $S(f)$ is given.
When $k\ge 2$, by definition, $k$-nondegeneracy of the singular point $p$
is equivalent to the $k$-nondegeneracy of 
the null vector field $\eta$ at $p$ (cf. \cite{SUY3}).
\end{remark}

\begin{proposition}\label{prop:indep}
 The $k$-th multiple contactness and $k$-nondegeneracy are both
 independent of the choice of an extension $\tilde \eta$ of $\eta$
 and also of the choice of admissible 
functions as in Definition \ref{def:admissible}. 
\end{proposition}
\begin{proof}
 We can take a local coordinate system $(U;x^1,\dots,x^n)$
 of $M^n$ such that $x^n=\phi$.
Write
 \[
\tilde \eta:=\sum_{j=1}^n c^j \frac{\partial}{\partial x^j},
 \]
 where $c^j$ $(j=1,\ldots,n)$ are $\K$-differentiable functions.
 Then we have that 
 $\phi'=\sum_{j=1}^n c^j \phi_{x^j}=c^n$.

 Let $\psi$ be another admissible 
function defined on $U$.
 Then 
 \[
    \psi'=
      \sum_{j=1}^n c^j \frac{\partial\psi}{\partial x^j}
       =c^n \frac{\partial\psi}{\partial x^n}
       =\phi'\frac{\partial\psi}{\partial x^n}.
 \]
 Thus $\psi'$ is proportional to $\phi'$. Then the assertion follows
 inductively.
\end{proof}

Corollary 2.5 in  \cite{SUY3} is now generalized
into the following assertion:

\begin{theorem}\label{thm:actual}
 Let $\tilde \eta$ be an extension of the vector field $\eta$. 
 Let us assume $1\leq k\leq n$.
 Then the vector field $\eta$ is $k$-nondegenerate at $p$,
 but $\eta$ does not meet $S$ with multiplicity $k+1$ at $p$
 if and only if
 \[
     \phi(p)=\phi'(p)=\dots =\phi^{(k-1)}(p)=0,\quad
     \phi^{(k)}(p)\ne 0,
 \]
 and the Jacobi matrix of $\K$-differentiable map
 \[
     \Lambda:=(\phi,\phi',\dots, \phi^{(k-1)}):U\longrightarrow \K^k
 \]
 is of rank $k$ at $p$, where 
$\phi$ is an 
admissible $\K$-differentiable function
and
 \[
    \phi^{(0)}:=\phi,\quad 
    \phi^{(1)}(=\phi'):=d\phi(\tilde \eta), \quad\dots,\quad 
    \phi^{(k)}:=d\phi^{(k-1)}(\tilde \eta).
  \]
\end{theorem}
The proof of this theorem is completely parallel to that of
Corollary 2.5 in  \cite{SUY3}.

\medskip
To prove Theorem~\ref{thm:A} by applying Theorem~\ref{thm:actual},
we shall review the criterion for $A_k$-singularities in \cite{SUY3}.
Let $U^n$ be a domain in $\K^n$, and consider a map
$\Phi:U^n\to \K^{m}$ where $m\ge n$.
A point $p\in U^{n}$ is called a {\it singular point\/} if
the rank of the differential map $d\Phi$ is less than $n$.
Suppose that the singular set $S(\Phi)$ of $\Phi$ consists of 
a $\K$-differentiable hypersurface $U^n$.
Then a vector field
$\eta$ along $S$ is called a {\it null vector field\/} if
$d\Phi(\eta)$ vanishes identically. 
In this paper, we consider the case $m=n$ or $m=n+1$.
If $m=n$, we define a $\K$-differentiable
function $\lambda:U^n\to \K$
by
\begin{equation}\label{eq:lambda-map}
     \lambda:=\det(\Phi_{x^1},\dots,\Phi_{x^n}).
\end{equation}
On the other hand, 
if $\Phi:U^n\to\K^{n+1}$ ($m=n+1$) and
$\nu$ is a non-vanishing $\K$-normal vector field
(for a definition, see \cite[Section 1]{SUY3})
we set
\begin{equation}\label{eq:lambda-fr}
     \lambda:=\det(\Phi_{x^1},\dots,\Phi_{x^n},\nu).
\end{equation}
Then the singular set $S(\Phi)$ of the map $\Phi$
coincides with the zeros of $\lambda$.
Recall that  $p\in S(\Phi)$ is called {\it nondegenerate} 
if $d\lambda(p)\ne \zv$ (see \cite{SUY3} and Remark \ref{rmk:add}). 
Both of two cases \eqref{eq:lambda-map} and \eqref{eq:lambda-fr},
the functions $\lambda$ are admissible near $p$, if $p$
is non-degenerate (cf.\ Definition~\ref{def:admissible}).
When 
$S(\Phi)$ consists of nondegenerate singular points,
then it is a hypersurface and there exists
a non-vanishing null vector field $\eta$ on $S(\Phi)$.
Such a vector field $\eta$ determined up to a multiplication of
non-vanishing $\K$-differentiable functions.
The following  assertion holds as seen in \cite{SUY3}.

\begin{fact}\label{eq:fact_suy}
 Suppose $m=n$ and $\Phi$ is a $C^\infty$-map
 {\rm(}resp.\ $m=n+1$ and $\Phi$ is a front{\rm)}.
 Then $\Phi$ has an $A_{k}$-Morin singularity 
 {\rm(}resp.\ $A_{k+1}$-singularity{\rm)} at $p\in M^n$ if and only if 
 $\eta$ is $k$-nondegenerate at $p$ but does not meet $S(\Phi)$ with
 multiplicity $k+1$ at $p$. 
 {\rm(}Here multiplicity $1$ means that 
 $\eta$ meets $S(\Phi)$ at $p$ transversally, and
$1$-nondegeneracy is an empty condition.{\rm)}
\end{fact}

As an application of the fact for $m=n$, we now give a proof of 
Theorem~\ref{thm:A}:
Let $F:M^{n}\to \K^{n+1}$ 
be an immersed $\K$-differentiable hypersurface.
Recall that a point  $p\in M^{n}$ is called 
a {\it nondegenerate inflection point\/} if
the derivative $dh$ of the local Hessian function $h$
(cf.~\eqref{eq:hessian1}) with respect to $F$ does not vanish at $p$.
Then the set $I(F)$ of inflection points consists of a hypersurface,
called the {\it inflectional hypersurface}, and
the function $h$ is an admissible function on a
neighborhood of $p$ in $M^n$.
A nondegenerate inflection point $p$ is called 
an {\it $A_{k+1}$-inflection point} of $F$
if the asymptotic vector field $\xi$ is $k$-nondegenerate at $p$ but
does not meet $I(F)$ with multiplicity $k+1$ at $p$.
\begin{proof}[Proof of Theorem~\ref{thm:A}]
 Let $\nu$ be a map given by \eqref{eq:nu}, 
 and $\G:M^n\to  P((\K^{n+1})^*)$ the affine 
Gauss map induced from 
 $\nu$ by \eqref{eq:G}.
 We set
 \[
    \mu:=\det(\nu_{x^1},\nu_{x^2},\dots,\nu_{x^n},\nu),
 \]
where \lq$\det$\rq\ is the determinant function of $(\K^{n+1})^*$
under the canonical identification $(\K^{n+1})^*\cong \K^{n+1}$,
and $(x^1,\dots,x^n)$ is a local coordinate system of $M^n$.
 Then the singular set $S(\G)$ of $\G$ is just the zeros of $\mu$.
 By Theorem~\ref{thm:actual} and Fact~\ref{eq:fact_suy},
 our criteria for $A_{k+1}$-inflection points
 (resp.\ $A_{k+1}$-singular points) are completely 
 determined by the pair $(\xi, I(F))$
 (resp.\ the pair $(\eta,S({\mathcal G}))$).
 Hence it is sufficient to
 show the following three assertions (1)--(3). 
\begin{it}
 \begin{enumerate}
  \item\label{ass:1} $I(F)=S(\mathcal G)$,
  \item\label{ass:2} For each $p\in I(F)$, $p$ is a nondegenerate
       inflection point of $F$ if and only if
       it is a nondegenerate singular point of $\G$.
  \item\label{ass:3} The asymptotic direction of each nondegenerate
       inflection point $p$ of $F$ is equal to
       the null direction of $p$ as a singular point of $\G$.
 \end{enumerate}
\end{it}
 Let  $H=\sum_{i,j=1}^n h_{ij}dx^i\, dx^j$ be the Hessian form of
 $F$. 
Then we have that
 \begin{equation}\label{eq:split}
  \pmt{%
   h_{11}& \dots & h_{1n}& *\\
   \vdots & \ddots & \vdots & \vdots \\
   h_{n1}& \dots & h_{nn} & *\\
   0 & \dots & 0 & \nu\cdot \trans\nu
  }
  =
  \pmt{%
   \nu_{x^1}\\ \vdots\\ \nu_{x^n}\\ \nu
  } 
  (F_{x^1},\dots,F_{x^n},\trans\nu),
 \end{equation}
where $\nu\cdot \trans\nu=\sum_{j=1}^{n+1}(\nu^j)^2$ and
$\nu=(\nu^1,...,\nu^n)$ as a row vector.
Here, we consider a vector in $\K^n$ (resp.\ in $(\K^n)^*$)
 as a column vector (resp.\ a row vector), and 
$\trans{(\cdot)}$ denotes the transposition.
 We may assume that
 $\nu(p)\cdot\trans\nu(p)\ne0$ by a suitable 
affine transformation of $\K^{n+1}$,
 even when $\K=\C$.
 Since the matrix $(F_{x^1},\dots,F_{x^n},\trans\nu)$ is regular,
 \ref{ass:1} and \ref{ass:2} follow by taking the determinant of
 \eqref{eq:split}.
 Also by \eqref{eq:split}, $\sum_{i=1}^n a_ih_{ij}=0$
 for all $j=1,\dots,n$ holds if and only if 
 $\sum_{i=1}^n a_i\nu_{x^i}=\zv$, which proves \ref{ass:3}. 
\end{proof}
\begin{proof}[Proof of Proposition~\ref{prop:A}]
 Similar to the proof of Theorem \ref{thm:A}, 
 it is sufficient to show the following properties,
 by virtue of Theorem~\ref{thm:actual}.
\begin{it}
\renewcommand{\theenumi}{$(\arabic{enumi}')$}
\renewcommand{\labelenumi}{$(\arabic{enumi}')$}
\begin{enumerate}
 \item\label{ass:1a}
      $S(F)=I(\G)$, that is, the set of singular points of $F$
      coincides with the set of inflection points of the affine 
Gauss map.
 \item\label{ass:2a}
      For each $p\in I({\G})$, $p$ is a nondegenerate inflection point
      if and only if it is a nondegenerate singular point
      of $F$.
 \item\label{ass:3a}
      The asymptotic direction of each nondegenerate inflection point 
      coincides with the null direction of $p$ as a singular point of 
      $F$.
 \end{enumerate}
\end{it}
 Since $\G$ is an immersion, \eqref{eq:split} implies that
 \begin{align*}
  I(\G) &= \{p\,;\,(F_{x^1},\dots,F_{x^n},\trans\nu) 
            \text{ are linearly dependent at $p$}\}\\
       &= \{p\,;\, \lambda(p)=0\}
        \qquad(\lambda:=\det(F_{x^1},\dots,F_{x^n},\trans\nu)).
 \end{align*}
 Hence we have \ref{ass:1a}.
 Moreover, $h=\det(h_{ij})=\delta\lambda$ holds, where $\delta$
 is a function on $U$ which never vanishes on a neighborhood of $p$.
 Thus \ref{ass:2a} holds.
 Finally, by \eqref{eq:split}, $\sum_{j=1}^n b_j h_{ij}=0$
 for $i=1,\dots,n$ if and only if $\sum_{j=1}^n b_j F_{x^j}=\zv$,
 which proves \ref{ass:3a}.
\end{proof}

\begin{example}[$A_2$-inflection points on cubic curves]
 Let $\gamma(t):=\trans{(x(t),y(t))}$
 be a $\K$-differentiable curve in $\K^2$.
 Then $\nu(t):=(-\dot y(t),\dot x(t))\in (\K^2)^*$
 gives a normal vector, and
\[
   h(t)=\nu(t)\cdot \ddot\gamma(t)=\det(\dot \gamma(t),
\ddot \gamma(t))
 \]
 is the Hessian function.
 Thus $t=t_0$ is an $A_2$-inflection point if and only if 
 \[
   \det(\dot \gamma(t_0),\ddot \gamma(t_0))=0,\qquad
   \det(\dot \gamma(t_0),\dddot \gamma(t_0))\ne 0.
 \]
 Considering $\K^2\subset P(\K^3)$ as an affine subspace,
 this criterion is available for curves in $P(\K^3)$.
 When $\K=\C$,
 it is well-known that non-singular cubic curves 
in $P(\C^3)$
have exactly nine inflection points which are all of $A_2$-type.
 One special singular cubic curve is
 $2y^2-3x^3=0$
 in $P(\C^3)$ with homogeneous coordinates $[x,y,z]$,
 which can be parameterized as
 $\gamma(t)=[\sqrt[3]{2}t^2,\sqrt{3}t^3,1]$.
 The image of the dual curve of $\gamma$ in $P(\C^3)$ is 
 the image of $\gamma$ itself,
 and $\gamma$ has an $A_2$-type singular point
$[0,0,1]$ and
 an $A_2$-inflection point $[0,1,0]$. 

 These two points are interchanged by the duality.
 (The duality of fronts is explained in Section~\ref{sec:dual}.)
\end{example}

\begin{example}[The affine Gauss map of an $A_4$-inflection point]
 Let $F:\K^3\to\K^4$ be a map defined by
 \[
    F(u,v,w)
    =
      \raisebox{0.3cm}{${}^t$}\!\!\!
        \left(
          w,u,v,
          -u^2-\dfrac{3v^2}{2}+uw^2+vw^3
          -\dfrac{w^4}{4}+\dfrac{w^5}{5}-\dfrac{w^6}{6}
        \right)
            \qquad (u,v,w\in \K).
 \]
 If we define $\G:\K^3\to P(\K^4)\cong P((\K^4)^{*})$ by
 \[
     \G(u,v,w) =
        [-2uw-3vw^2+w^3-w^4+w^5,2u-w^2,3v-w^3,1]
 \]
 using the homogeneous coordinate system,
$\G$ gives the  affine Gauss map of $F$.
 Then the Hessian $h$ of $F$ is
 \[
   \det
   \pmt{
      -2&0&2w\\
       0&-3&3w^2\\
       2w&3w^2&2u+6vw-3w^2+4w^3-5w^4
   }
   =6(   2u+6vw-w^2+4w^3-2w^4).
 \]
 The asymptotic vector field is $\xi=(w,w^2,1)$.
 Hence we have
 \begin{align*}
  h&=6(2u+6vw-w^2+4w^3-2w^4),\\
  h'&=12(3v+6w^2-w^3),\qquad
  h''=144w,\qquad 
  h'''=144,
 \end{align*}
 where $h'=dh(\xi)$, $h''=dh'(\xi)$ and $h'''=dh''(\xi)$.
 The Jacobi matrix of
 $(h,h',h'')$ at $\zv$ is
 \[
    \pmt{%
     2 & * & * \\ 
     0 & 36 & * \\ 
     0 & 0 & 144
     }.
 \]
 This implies that $\xi$  is $3$-nondegenerate at $\zv$
 but does not meet $I(F)=h^{-1}(0)$  at $p$ with multiplicity 
 $4$, that is,
 $F$ has an $A_{4}$-inflection point at $\zv$.
 On the other hand, 
 $\G$ has the ${A}_3$-Morin singularity  at $\zv$.
 In fact, by the coordinate change
 \[
       U=2u-w^2,\qquad V=3v-w^3,\qquad W=w,
 \]
 it follows that $\G$ is represented by a map germ
 \[
    (U,V,W)\longmapsto
            -(UW+VW^2+W^4,U,V).
 \]
 This coincides with the typical ${A}_3$-Morin singularity
 given in (A.3) in \cite{SUY3}.
\end{example}

\section{Duality of wave fronts}
\label{sec:dual}

Let $P(\K^{n+2})$ be the $(n+1)$-projective space over $\K$.
We denote by $[x]\in P(\K^{n+2})$ the projection of a vector
$x=\trans(x^0,\dots,x^{n+1})\in \K^{n+2}\setminus\{\boldsymbol{0}\}$.
Consider a $(2n+3)$-submanifold of $\K^{n+2}\times (\K^{n+2})^*$
defined by
\[
  \widetilde C:=
       \{(x,y)\in \K^{n+2}\times (\K^{n+2})^*\,;\, x\cdot y=0\},
\]
and also a $(2n+1)$-submanifold of 
$P(\K^{n+2})\times P((\K^{n+2})^*)$
\[
  C:=\{([x],[y])\in P(\K^{n+2})\times P((\K^{n+2})^*)
\,;\, x\cdot y=0\}.
\]
As $C$ can be canonically identified with the projective tangent
bundle $PTP(\K^{n+2})$,
it has a canonical contact structure:
Let $\pi:\widetilde C\to C$ be the canonical projection,
and define a $1$-from
\[
  \omega:=\sum_{j=0}^{n+1} \left(x^jdy^j-y^j d{x^j}\right),
\]
which is considered as a $1$-form of $\widetilde C$.
The tangent vectors of the curves $t\mapsto (tx,y)$ and $t\mapsto (x,ty)$
at $(x,y)\in \widetilde C$ generate the kernel of $d\pi$.
Since these two vectors also belong to the kernel of $\omega$ 
and $\dim(\ker\omega)=2n+2$,
\[
   \Pi:=d\pi(\ker \omega)
\]
is a $2n$-dimensional vector subspace of $T_{\pi(x,y)}C$.
We shall see that $\Pi$ is the contact structure on $C$.
One can check that it coincides with the canonical contact structure 
of $PTP(\K^{n+2})$ ($\cong C$).
Let $U$ be an open subset of $C$ and $s:U\to \K^{n+2}
\times (\K^{n+2})^*$
a section of the fibration $\pi$.
Since $d\pi\circ ds$ is the identity map,  it can be easily checked
that $\Pi$ is contained in the kernel of the $1$-form $s^*\omega$.
Since $\Pi$ and the kernel of the $1$-form $s^*\omega$ are the same
dimension, they coincide.
Moreover, suppose that $p=\pi(x,y)\in C$ satisfies $x^i\ne 0$ and
$y^j\ne 0$.
We then consider a map of $\K^{n+1}\times (\K^{n+1})^*
\cong \K^{n+1}\times \K^{n+1}$ into 
$\K^{n+2}\times(\K^{n+2})^*\cong \K^{n+2}\times \K^{n+2}$
defined by
\[
 (a^0,\dots,a^{n},b^0,\dots,b^n)
  \mapsto 
  (a^0,\dots,a^{i-1},1,a^{i+1},\dots,a^n,b^0,\dots,b^{j-1},1,b^{j+1},\dots,b^n),
\]
and denote by $s_{i,j}$ the restriction of the map to the
neighborhood of $p$ in $C$.
Then one can easily check that
\[
  s_{i,j}^*
  \left[
       \omega\wedge \left(\bigwedge\nolimits^{n}d\omega\right)
  \right]
\]
does not vanish at $p$.
 Thus $s_{i,j}^*\omega$ is a contact form, and the 
hyperplane field $\Pi$ defines
a canonical contact structure on $C$.
Moreover, the two projections from $C$
into $P(\K^{n+2})$ are both Legendrian fibrations,
namely we get a double Legendrian fibration.
Let $f=[F]:M^{n}\to P(\K^{n+2})$ be a front.
Then there is a Legendrian immersion of the form
$L=([F],[G]):M^{n}\to C$.
Then $g=[G]:M^{n}\to P((\K^{n+2})^*)$
satisfies \eqref{eq:FG1} and \eqref{eq:FG2}.
In particular, $L:=\pi(F,G):M^{n}\to C$ gives a Legendrian immersion,
and $f$ and $g$ can be regarded as mutually dual
wave fronts as projections of $L$.

\begin{proof}[Proof of Theorem~\ref{thm:B}]
 Since our contact structure on $C$ can be identified with
 the contact structure on the projective tangent bundle
 on $P(\K^{n+2})$, we can apply the criteria
 of $A_k$-singularities as in Fact \ref{eq:fact_suy}.
 Thus a nondegenerate singular point $p$ is an $A_k$-singular point of
 $f$ if and only
 if the null vector field $\eta$ of $f$ (as a wave front)
 is $(k-1)$-nondegenerate
 at $p$, but does not meet the 
hypersurface $S(f)$ with multiplicity $k$ at $p$.
 Like as in the proof of Theorem~\ref{thm:A},
 we may assume that
 $\trans F(p)\cdot F(p)\ne0$ and $G(p)\cdot \trans G(p)\ne0$
 simultaneously by a suitable affine transformation of $\K^{n+2}$,
 even when $\K=\C$.
 Since $(F_{x^1},\dots,F_{x^n}, F, \trans G)$ is a regular
 $(n+2)\times (n+2)$-matrix
 if and only if $f=[F]$ is an immersion, 
 the assertion immediately follows from the identity
 \begin{equation}\label{eq:split2}
  \pmt{
    h_{11}& \dots & h_{1n}& 0 & *\\
    \vdots & \ddots & \vdots & \vdots & \vdots  \\
    h_{n1}& \dots & h_{nn} & 0 & *\\
    0 & \dots & 0 & 0 & G\cdot \trans G \\
    * & \dots & * & \trans F\cdot F & 0
   }=
   \pmt{
     G_{x^1}\\ \vdots\\ G_{x^n}\\ G \\ \trans F
   } 
   (F_{x^1},\dots,F_{x^n}, F, \trans G).
 \end{equation}
\end{proof}

\begin{proof}[Proof of Theorem~\ref{thm:C}]
 Let $g:M^2\to P((\R^4)^*)$ be the dual of $f$.
We fix $p\in M^2$ and take a simply connected
 and connected neighborhood $U$ of $p$.

 Then there are lifts $\hat f,\hat g:U\to S^3$ into the unit sphere $S^3$
 such that 
 \[
    \hat f\cdot \hat g=0,\qquad 
    d\hat f(v)\cdot \hat g=d\hat g(v)\cdot \hat f=0 \quad (v\in TU),
 \]
 where $\cdot$ is the canonical inner product on $\R^4\supset S^3$.
 Since $\hat f\cdot \hat f=1$,
 we have
 \[
      d\hat f(v)\cdot \hat f(p)
          =0\qquad (v\in T_pM^2).
 \]
 Thus
 \[
       d\hat f(T_pM^2)=\{\zeta \in S^3\,;\,
       \zeta\cdot \hat f(p)=\zeta \cdot \hat g(p)=0\},
 \]
 which implies that $df(TM^2)$ is  equal to  the limiting tangent bundle
 of the  front $g$.
 So we apply (2.5) in \cite{SUY2} for $g$:
Since the singular set $S(g)$ of $g$
consists only of cuspidal edges and swallowtails,
the Euler number of $S(g)$ vanishes.
Then it holds that
 \[
    \chi(M_+)+\chi(M_-)=\chi(M^2)=\chi(M_+)-\chi(M_-)
           +i_2^+(f)-i_2^-(f),
 \]
 which proves the formula.
\end{proof}

When $n=2$, the duality of fronts in the unit $2$-sphere $S^2$
(as the double cover of $P(\R^3)$) plays a crucial role 
for obtaining the
classification theorem in
\cite{MU}
for complete flat fronts with embedded ends in $\R^3$.
Also, a relationship between the number of inflection points 
and the number of double tangents on certain class of simple closed
regular curves in $P(\R^3)$ is given in \cite{tu3}.
(For the geometry and a duality of fronts in $S^2$, see \cite{A}.)
In \cite{porteous}, Porteous investigated the duality between 
$A_k$-singular points and $A_k$-inflection points when $k=2,3$
on a surface in $S^3$.

\section{Cuspidal curvature on $3/2$-cusps}
\label{sec:cusp}

Relating to the duality between singular points
and inflection points, we introduce a curvature
on $3/2$-cusps of planar curves:

Suppose that $(M^2,g)$ is an oriented Riemannian manifold,
$\gamma:I\to M^2$ is a front, $\nu(t)$ is a unit normal vector
field, and $I$ an open interval.
Then $t=t_0\in I$ is a $3/2$-cusp if and only if $\dot\gamma(t_0)=\zv$
and $\Omega(\ddot\gamma(t_0),\dddot\gamma(t_0))\ne 0$,
where $\Omega$ is the unit $2$-form on $M^2$, that is, the Riemannian
area element, 
and the dot means the covariant derivative. 
When $t=t_0$ is a $3/2$-cusp, $\dot\nu(t)$ 
does not vanish
(if $M^2=\R^2$, it follows from Proposition $\mathrm A'$). 
Then we take the (arclength) parameter $s$ near $\gamma(t_0)$
so that  $|\nu'(s)|=\sqrt{g(\nu'(s),\nu'(s))}=1$
($s\in I$), where  $\nu'=d\nu/ds$.
Now we define the {\it cuspidal curvature} $\mu$ by
\[
 \mu:=
     2\operatorname{sgn}(\rho)\left. \sqrt{\left |{ds}/{d\rho}\right |}
     \right|_{s=s_0} \qquad (\rho=1/\kappa_g),
\]
where 
we choose the unit normal $\nu(s)$ 
so that it is smooth around $s=s_0$
($s_0=s(t_0)$).
If $\mu>0$ (resp. $\mu<0$), 
the cusp is called {\it positive} (resp.\ {\it negative}).
It is interesting phenomenon that the 
left-turning cusps have negative
cuspidal curvature, although the left-turning regular
curves have positive geodesic curvature (see Figure \ref{cusps}).
%%%%%%%%%%%%%%%%%%%%%%%%%%%%%%%%%%%%%
\begin{figure}
\begin{center}
        \includegraphics[width=4cm]{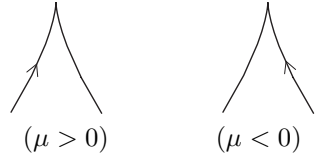}
\centerline{$(\mu>0)\qquad \qquad (\mu<0)$}
\caption{a positive cusp and a negative cusp}\label{cusps}
\end{center}
\end{figure}
%%%%%%%%%%%%%%%%%%%%%%%%%%%%%%%%%%%%%%
Then it holds that
\begin{equation}\label{eq:general3}
 \mu=    \left.\frac{\Omega(\ddot \gamma(t),\dddot \gamma(t))}
     {|\ddot \gamma(t)|^{5/2}}\right|_{t=t_0}
    = 
    \left. 2\frac{\Omega(\nu(t),\dot \nu(t))}{%
         \sqrt{|\Omega(\ddot\gamma(t),\nu(t))|}}
\right|_{t=t_0}.
\end{equation}
We now examine the case that $(M^2,g)$ is the Euclidean plane $\R^2$,
where $\Omega(v,w)$ ($v,w\in \R^2$) coincides with the determinant 
$\det(v,w)$
of
the $2\times 2$-matrix $(v,w)$.
A {\it cycloid} is a rigid motion of the
curve given by
$c(t):=a(t-\sin t,1-\cos t)$ ($a>0$),
and here $a$ is called the {\it radius\/} of the cycloid.
The cuspidal curvature of $c(t)$ at $t\in 2\pi \mathbf Z$
is equal to $-1/\sqrt{a}$.
In \cite{ume}, the second author proposed to consider
the  curvature as the inverse of radius of the cycloid
which gives the best approximation of the given $3/2$-cusp.
As shown in the next proposition, $\mu^2$
attains this property:
\begin{proposition}
 Suppose that $\gamma(t)$ has a $3/2$-cusp at $t=t_0$.
 Then by a suitable choice of the parameter $t$, 
 there exists a unique cycloid $c(t)$ such that
 \[
       \gamma(t)-c(t)=o((t-t_0)^3),
 \]
 where $o((t-t_0)^3)$ denotes a  higher order term than $(t-t_0)^3$.
 Moreover, the square of the absolute value of cuspidal 
curvature of $\gamma(t)$ at
 $t=t_0$ is equal to the inverse of the radius of the cycloid $c$.
\end{proposition}

\begin{proof}
 Without loss of generality, we may set $t_0=0$ and $\gamma(0)=\zv$.
 Since $t=0$ is a singular point, 
 there  exist smooth functions $a(t)$ and $b(t)$  such that
 $\gamma(t)=t^2(a(t),b(t))$.
 Since $t=0$ is a $3/2$-cusp, $(a(0),b(0))\ne \zv$.
 By a suitable rotation of $\gamma$, we may assume that $b(0)\ne 0$ and
 $a(0)=0$.
 Without loss of generality, we may assume that  $b(0)>0$. 
 By setting $s=t\sqrt{b(t)}$,
 $\gamma(s)=\gamma(t(s))$ has the expansion
 \[
   \gamma(s)=(\alpha s^3, s^2)+o(s^3) \qquad (\alpha\ne 0).
 \]
 Since the cuspidal curvature changes sign by reflections on $\R^2$,
 it is sufficient to consider the case $\alpha>0$.
 Then, the cycloid
 \[
    c(t):=\frac{2}{9\alpha^2}(t-\sin t,1-\cos t)
 \]
 is the desired one by setting $s=t/(3\alpha)$.
\end{proof}
It is well-known that the cycloids are
the solutions of the brachistochrone problem.
We shall propose to call the number $1/|\mu|^2$ the 
{\it cuspidal curvature radius} which corresponds the 
radius of the best approximating cycloid $c$. 

\begin{remark}
During
the second author's stay at Saitama University,
Toshizumi Fukui pointed out the followings:
Let $\gamma(t)$ be a regular curve in $\R^2$ with
non-vanishing curvature function $\kappa(t)$.
Suppose that $t$ is the arclength parameter of $\gamma$.
For each $t=t_0$, there exists a unique cycloid
$c$ such that a point on $c$ gives the best
approximation of $\gamma(t)$ at $t=t_0$ (namely
$c$ approximates $\gamma$ up to the third jet at $t_0$).
The angle $\theta(t_0)$ between the axis
(i.e. the normal line of $c$ at the singular points)
 of the cycloid and the normal line of $\gamma$ at $t_0$ is 
given by
\begin{equation}\label{c:theta}
\sin \theta=\frac{\kappa^2}{\sqrt{\kappa^4+\dot\kappa^2}},
\end{equation}
and the radius $a$ of the cycloid is given by
\begin{equation}\label{c:k}
a:=\frac{\sqrt{\kappa^4+\dot\kappa^2}}{|\kappa|^3}.
\end{equation}
One can prove \eqref{c:theta} and \eqref{c:k} by straightforward
calculations. The cuspidal curvature radius can be considered as
the limit.
\end{remark}

\begin{added}
 In a recent authors' preprint,
 ``The intrinsic duality of wave fronts (arXiv:0910.3456)'',
 $A_{k+1}$-singularities are defined intrinsically.
 Moreover, the duality between fronts and their
 Gauss maps is also explained intrinsically.
\end{added}


\begin{thebibliography}{99}
\bibitem{A}
V. I. Arnold,
 {\it The geometry of spherical curves
      and the algebra of quaternions},
 Uspekhi Mat. Nauk {\bf 50}(1995) 3-68,
 English transl. in Russian Math. Surveys {\bf 50} (1995).

\bibitem{BGM}
T. Banchoff, T. Gaffney and C. McCrory,
 {\sc  Cusps of Gauss Mappings},
 Pitman publ. (1982).


\bibitem{BW}
D. Bleecker and L. Wilson,
 {\it Stability of Gauss maps},
 Illinois J. of Math. {\bf 22} (1978)
279--289.

         
\bibitem{ips}
S. Izumiya, D. Pei and T. Sano,
 {\itshape Singularities of hyperbolic Gauss maps},
     Proc. London Math. Soc. {\bf 86} (2003), 485--512


\bibitem{iz}
S. Izumiya,
  {\itshape Legendrian dualities and spacelike hypersurfaces
             in the lightcone},
Moscow Math. J. {\bf 9} (2009), 325--357.

\bibitem{MU}
  S. Murata and M. Umehara,
  {\itshape 
   Flat surfaces with singularities in Euclidean $3$-space}, 
   J. of Differential Geometry, {\bf 82} (2009), 279--316.

         
\bibitem{porteous}
I. R. Porteous,
   {\itshape 
      Some remarks on duality in $S^3$},
      Geometry and topology of caustics, Banach Center Publ. {\bf 50},
      Polish Acad. Sci., Warsaw, (2004), 217--226.
         
\bibitem{SUY1} 
K. Saji, M. Umehara and K. Yamada,
        {\itshape The geometry of fronts}, 
Ann. of Math. {\bf 169} (2009), 491--529.
 

\bibitem{SUY2} 
\bysame,  
   {\itshape Behavior of corank one singular points on wave fronts},
Kyushu J. of Mathematics {\bf 62} (2008), 259--280.

\bibitem{SUY3} 
\bysame, 
  {\itshape $A_k$-singularities of wave fronts},
Math. Proc.
     Cambridge Philosophical Soc.,
{\bf 146} (2009), 731-746. 
        
\bibitem{tu3}
G. Thorbergsson and M. Umehara,
   {\itshape  Inflection points and double tangents 
              on anti-convex curves in the real projective plane}, 
    T\^ohoku Math. J. {\bf 60} (2008), 149--181. 
%
\bibitem{ume}
  M. Umehara,
  \newblock{\em Differential geometry on surfaces with singularities},
  \newblock (in Japanese), The world of Singularities 
  (ed. H. Arai, T. Sunada and K. Ueno)
  Nippon-Hyoron-sha Co., Ltd. (2005), 50--64.
%
\bibitem{uribe} 
R. Uribe-Vargas, 
 {\itshape A projective invariant for swallowtails
           and godrons, and global theorems on the flecnodal curve},
  Moscow Math. J. {\bf 6}, (2006), 731--768.
\end{thebibliography}
\end{document}